\documentclass[11pt]{article}
\usepackage[english]{babel}
\usepackage{times}
\usepackage{draftcopy}
\usepackage{amsmath,amstext,amssymb,amsthm, amsfonts}
\usepackage[dvips]{graphics,color}
\newcommand{\bi}{\begin{itemize}}  
\newcommand{\ei}{\end{itemize}}     
\newcommand{\bc}{\begin{center}}  
\newcommand{\ec}{\end{center}}     

\newcommand{\ls}[1]
   {\dimen0=\fontdimen6\the\font \lineskip=#1\dimen0
   \advance\lineskip.5\fontdimen5\the\font \advance\lineskip-\dimen0
   \lineskiplimit=.9\lineskip \baselineskip=\lineskip
   \advance\baselineskip\dimen0 \normallineskip\lineskip
   \normallineskiplimit\lineskiplimit \normalbaselineskip\baselineskip
   \ignorespaces }

\numberwithin{equation}{section}

\newtheorem{theorem}{Theorem}[section]

\theoremstyle{definition}

\newcommand{\R}{\mathbb{R}}

\newcommand{\I}{\mathbf{I}}
\newcommand{\K}{\mathcal{K}}

\newcommand{\X}{\mathbb{X}}
\newcommand{\loc}{{\rm loc}}

\newcommand{\vsup} {\mathop{\rm lim\,sup}}
\newcommand{\isup} {\mathop{\rm lim\,inf}}
\newcommand{\slim} {\mathop{\rm lim\,sup}}
\newcommand{\ilim} {\mathop{\rm lim\,inf}}
\newcommand{\dist}{\operatorname{dist}}

\title{Uniform Global Attractors for Non-Autonomous Dissipative Dynamical Systems}
\begin{document}
\maketitle

\centerline{\scshape Michael Zgurovsky}
\medskip
{\footnotesize
 \centerline{National Technical University of Ukraine
``Kyiv Polytechnic Institute'', }
   \centerline{ Peremogy ave., 37, build, 1, 03056,
Kyiv, Ukraine,}
} 

\medskip

\centerline{\scshape Mark Gluzman}
\medskip
{\footnotesize
 \centerline{ Department of Applied Physics and Applied Mathematics, Columbia University,}
   \centerline{New York, NY 10027, USA}
}

\medskip

\centerline{\scshape Nataliia Gorban}
\medskip
{\footnotesize
 \centerline{ Institute for Applied System Analysis,
National Technical University of Ukraine ``Kyiv Polytechnic
Institute'', }
   \centerline{Peremogy ave., 37, build, 35, 03056, Kyiv, Ukraine}
  }
\medskip

\centerline{\scshape Pavlo Kasyanov {$^*$}}
\medskip
{\footnotesize
\centerline{ Institute for Applied System Analysis,
National Technical University of Ukraine ``Kyiv Polytechnic
Institute'', }
   \centerline{Peremogy ave., 37, build, 35, 03056, Kyiv, Ukraine}
  }

   \medskip

\centerline{\scshape Liliia Paliichuk}
\medskip
{\footnotesize
\centerline{ Institute for Applied System Analysis,
National Technical University of Ukraine ``Kyiv Polytechnic
Institute'', }
   \centerline{Peremogy ave., 37, build, 35, 03056, Kyiv, Ukraine}
  }

 \medskip

\centerline{\scshape Olha Khomenko}
\medskip
{\footnotesize
 \centerline{ Institute for Applied System Analysis,
National Technical University of Ukraine ``Kyiv Polytechnic
Institute'', }
   \centerline{Peremogy ave., 37, build, 35, 03056, Kyiv, Ukraine}
  }

\bigskip


\begin{abstract}
In this paper we consider sufficient conditions for the existence of uniform compact global attractor for non-autonomous dynamical systems in special classes of infinite-dimensional phase spaces. The obtained generalizations allow us to avoid the restrictive compactness assumptions on the space of shifts of non-autonomous terms in particular evolution problems. The results are applied to several evolution inclusions.
\end{abstract}

\section{Introduction}

The standard scheme of investigation of uniform the long-time behavior for all solutions of non-autonomous problems
covers non-autonomous problems of the form
\begin{equation}\label{eq:8}
\partial_t
u(t)=A_{\sigma(t)}(u(t)),
\end{equation}
where $\sigma(s),$ $s\ge 0$, is a functional parameter called the time symbol of equation (\ref{eq:8}) ($t$ is replaced by $s$). In applications to
mathematical physics equations, a function $\sigma(s)$ consists of all time-dependent terms of the equation under consideration: external forces,
parameters of mediums, interaction functions, control functions, etc; Chepyzhov and Vishik \cite{VCh0,VCh,VCh1}; Sell \cite{Sell}; Zgurovsky et al.
\cite{ZMK3} and references therein; see also Hale \cite{Hale}; Ladyzhenskaya \cite{Lad3}; Mel'nik and Valero \cite{MeVa2000}; Iovane, Kapustyan and
Valero \cite{IO}. In the mentioned above papers and books it is assumed that the symbol $\sigma$ of equation (\ref{eq:8}) belongs to a Hausdorff topological
space $\Xi_+$ of functions defined on $\R_+$ with values in some complete metric space. Usually, in applications, the topology in the space $\Xi_+$ is a
local convergence topology on any segment $[t_1,t_2]\subset \R_+$. Further, they consider the family of equations (\ref{eq:8}) with various symbols
$\sigma(s)$ belonging to a set $\Sigma\subseteq\Xi_+$. The set $\Sigma$ is called the symbol space of the family of equations (\ref{eq:8}). It is assumed
that the set $\Sigma$, together with any symbol $\sigma(s)\in \Sigma$, contains all positive translations of $\sigma(s)$:
$\sigma(t+s)=T(t)\sigma(s)\in\Sigma$ for any $t,s\ge 0$. The symbol space $\Sigma$ is invariant with respect to the translation semigroup
$\{T(t)\}_{t\ge0}$: $T(t)\Sigma\subseteq\Sigma$ for any $t\ge 0$. To prove the existence of uniform trajectory attractors they suppose that the symbol space
$\Sigma$ with the topology induced from $\Xi_+$ is a compact metric space. Mostly in applications, as a symbol space $\Sigma$ it is naturally to consider
the hull of translation-compact function $\sigma_0(s)$ in an appropriate Hausdorff topological space $\Xi_+$. The direct realization of this approach to differential-operator inclusions, PDEs with Caratheodory's nonlinearities, optimization problems, etc, is problematic without any additional assumptions for parameters of Problem (\ref{eq:8}) and requires the translation-compactness of
the symbol $\sigma(s)$ in some compact Hausdorff topological space of measurable multivalued mappings acts from $\R_+$ to some metric space of
operators from $(V\to 2^{V^*}),$ where $V$ is a Banach space and $V^*$ is its dual space,  satisfying (possibly) only growth and sign assumptions. To avoid this technical difficulties we present an alternative
approach for the existence and construction of the uniform global attractor for classes of non-autonomous dynamical systems in special classes of infinite-dimensional phase spaces.

\section{Main Constructions and Results}

Let $p\ge 2$ and $q> 1$ be such that $\frac1p+\frac1q=1$, $(V;H;V^*)$ to be evolution triple
such that $V\subset H$ with compact embedding. For each $t_1,t_2\in \mathbb{R}$, $0\le
t_1<t_2<+\infty$, consider the space 
\[
W_{t_1,t_2}:=\{y(\cdot)\in L_{p}(t_1,t_2;V)\, :\, y'(\cdot)\in
L_{q}(t_1,t_2;{V^*})\},
\]
where $y'(\cdot)$ is a derivative of an element $y(\cdot)\in
L_{p}(t_1,t_2;V)$ in the sense of distributions
$\mathcal{D}^*([t_1,t_2];{V^*})$. The space $W_{t_1,t_2}$ endowed with
the norm
\[
\|y\|_{W_{t_1,t_2}}:=\|y\|_{L_{p}(t_1,t_2;V)}+\|y'\|_{L_{q}(t_1,t_2;{V^*})}, 
\quad y\in W_{t_1,t_2},
\]
is a reflexive Banach space. Note that $W_{t_1,t_2}\subset C([t_1,t_2];H)$ with continuous and dense embedding; Gajewsky et al \cite[Chapter~IV]{GGZ}.
For each $\tau\ge 0$, consider the Fr\'echet space 
\[
W^{\loc}([\tau,+\infty)):=\{y:[\tau,+\infty)\to H\,:\, \Pi_{t_1,t_2}y\in W_{t_1,t_2}\mbox{ for each }[t_1,t_2]\subset[\tau,+\infty)\},
\]
where $\Pi_{t_1,t_2}$ is the restriction operator to the finite time interval $[t_1,t_2]$. We recall that the sequence
$\{f_n\}_{n\ge 1}$ converges in $W^{\rm \rm loc}([\tau,+\infty))$ (in $C^{\rm \rm loc}([\tau,+\infty);H)$ respectively) to $f\in W^{\rm \rm loc}([\tau,+\infty))$ (to $f\in C^{\rm \rm loc}([\tau,+\infty);H)$ respectively) as $n\to+\infty $ if and only if the
sequence $\{\Pi_{t_1,t_2}f_n\}_{n\ge 1}$ converges in $W_{t_1,t_2}$ (in $C([t_1,t_2];H)$ respectively) to $\Pi_{t_1,t_2}f$ as $n\to+\infty$ for each finite time interval
$[t_1,t_2]\subset [\tau,+\infty)$. Further we denote that
\[
T(h)y(\cdot)=\Pi_{0,+\infty}\,y(\,\cdot\,+h), \quad y\in W^{\rm \rm loc}(\mathbb{R}_+), \ h \ge
0,
\]
where $\R_+=[0,+\infty)$ and $\Pi_{0,+\infty}$ is the restriction operator to the time interval $[0,+\infty)$.

Throughout the paper we consider the \textit{family of solution sets} $\{\K_\tau^+\}_{\tau\ge 0}$ such that $\K_{\tau}^+\subset W^{\loc}([\tau,+\infty))$ for each $\tau\ge 0$ and $\K_{\tau_0}^+\ne\emptyset$ for some $\tau_0\ge 0$. In the most of applications as $\K_\tau^+$ can be considered the family of globally defined on
$[\tau,+\infty)$ weak solutions for particular non-autonomous evolution problem (see Section~\ref{sec:app}).

 To state the main assumptions on the {family of solution sets} $\{\K_\tau^+\}_{\tau\ge 0}$ it is necessary to formulate two auxiliary definitions.

A function $\varphi\in L_\gamma^{\rm \rm loc}(\R_+)$, $\gamma>1$, is called  \textit{translation bounded} function in $L_\gamma^{\rm \rm loc}(\R_+)$ if
\[\sup_{t\ge0}\int_{t}^{t+1}|\varphi(s)|^\gamma ds<+\infty;\] Chepyzhov and Vishik \cite[p.~105]{VChRI}. A function $\varphi\in L_1^{\rm \rm
loc}(\R_+)$ is called \textit{a translation uniform integrable (t.u.i.)} function in $L_1^{\rm \rm loc}(\R_+)$ if
\[
\lim_{K\to+\infty}\sup_{t\ge0}\int_{t}^{t+1}|\varphi(s)|\I{\{|\varphi(s)|\ge K\}}ds=0;
\]
Gorban et al. \cite{GKK}.
Note that Dunford-Pettis compactness criterion provides that  $\varphi\in L_1^{\rm \rm loc}(\R_+)$ is a t.u.i. function in $L_1^{\rm \rm loc}(\R_+)$ if and
only if for every sequence of elements $\{\tau_n\}_{n\ge 1}\subset\R_+$, the sequence $\{\varphi(\,\cdot\,+\tau_n)\}_{n\ge 1}$ contains a subsequence
converging weakly in $L_1^{\rm \rm loc}(\R_+)$. Note that for each $\gamma>1$, every translation bounded in $L_{\gamma}^{\rm \rm loc}(\R_+)$ function is
t.u.i.  in $L_1^{\rm \rm loc}(\R_+)$; Gorban et al. \cite{GKK}.
\medskip

\noindent\textbf{Main assumptions.} Let the following two assumptions hold:
\begin{itemize}
\item[(A1)] there exist a t.u.i. in $L_1^{\rm \rm loc}(\R_+)$ function $c_1:\R_+\to\R_+$ and a constant $\alpha_1>0$ such that for each $\tau\ge 0$, $y\in \K_\tau^+$, and $t_2\ge t_1 \ge \tau$, the following inequality holds:
\begin{equation}\label{eq:1a}
\|y(t_2)\|_H^2-\|y(t_1)\|_H^2+\alpha_1\int_{t_1}^{t_2}\|y(t)\|_{V}^{p}dt\le \int_{t_1}^{t_2}c_1(t)dt;
\end{equation}
\item[(A2)] there exist a t.u.i. in $L_1^{\rm \rm loc}(\R_+)$ function $c_2:\R_+\to\R_+$ and a constant $\alpha_2>0$ such that
for each $\tau\ge 0$, $y\in \K_\tau^+$, and $t_2\ge t_1 \ge \tau$, the following inequality holds:
\begin{equation}\label{eq:1b}
\int_{t_1}^{t_2}\|y'(t)\|_{{V^*}}^{q}dt\le \alpha_2 \int_{t_1}^{t_2}\|y(t)\|_{V}^{p}dt+ \int_{t_1}^{t_2}c_2(t)dt.
\end{equation}
\end{itemize}

To characterize the uniform long-time behavior of solutions for non-autonomous dissipative dynamical system consider the \textit{united trajectory space}
$\mathcal{K}_\cup^+$ for the family of solutions $\{\K_\tau^+\}_{\tau\ge 0}$ shifted to zero:
\begin{equation}\label{eq:2}
\mathcal{K}_\cup^+:=\bigcup_{\tau\ge 0}\left\{T(h)y(\,\cdot + \tau\,)\,:\, y(\,\cdot\,)\in \mathcal{K}_{\tau}^+,\, h\ge 0\right\}\subset W^{\rm loc}(\R_+),
\end{equation}
and the \textit{extended united trajectory space} for the family $\{\K_\tau^+\}_{\tau\ge 0}$:
\begin{equation}\label{eq:3}
\quad \mathcal{K}^+:={\rm cl}_{C^{\rm \rm loc}(\mathbb{R}_+;H)}\left[\mathcal{K}_\cup^+\right],
\end{equation}
where ${\rm cl}_{C^{\rm \rm loc}(\mathbb{R}_+;H)}[\,\cdot\,]$ is the closure in $C^{\rm \rm loc}(\mathbb{R}_+;H)$.
Since
$T(h)\mathcal{K}_\cup^+\subseteq\mathcal{K}_\cup^+$  for each $h\ge0$, then
\begin{equation}\label{eq:4}
T(h)\mathcal{K}^+\subseteq\mathcal{K}^+\mbox{ for each }h\ge0,
\end{equation}
due to
\[
\rho_{C^{\rm \rm loc}(\mathbb{R}_+;H)}(T(h)u,T(h)v)\le \rho_{C^{\rm \rm loc}(\mathbb{R}_+;H)}(u,v)\mbox{ for each }u,v\in C^{\rm \rm loc}(\mathbb{R}_+;H),
\]
where $\rho_{C^{\rm \rm loc}(\mathbb{R}_+;H)}$ is the standard metric on Fr${\rm\acute{e}}$chet space $C^{\rm \rm loc}(\mathbb{R}_+;H)$. Therefore the set
\begin{equation}\label{eq:6}
\X:=\{y(0) \,:\, y\in \mathcal{K}^+ \}
\end{equation}
is closed in $H$ (it follows from Theorem~\ref{teor:2}). We endow this set $\X$ with metric
\[
\rho_\X(x_1,x_2)=\|x_1-x_2\|_{H},\quad x_1,x_2\in \X.
\]
Then we obtain that ($\X$,$\rho$) is a Polish
space (complete separable metric space).

Let us define the multivalued semiflow (\emph{m-semiflow}) $G:\R_+\times \X\to 2^{\X}$:
\begin{equation}\label{eq:5} G(t,y_0):=\{y(t)\,:\,
y(\cdot)\in \mathcal{K}^+\mbox{ and }y(0)=y_0 \}, \quad t\ge 0,\,y_0\in \X.
\end{equation}

According to  (\ref{eq:4}), (\ref{eq:6}), and (\ref{eq:5}) for each $t\ge 0$ and $y_0\in \X$ the set $G(t,y_0)$ is nonempty.
Moreover, the following two conditions hold:
\begin{itemize}
\item[(i)] $G\left( 0,\cdot \right) =I$ is the identity map; \item[(ii)] $G\left( t_1+t_2,y_0\right) \subseteq
G\left( t_1,G\left( t_2,y_0\right) \right) ,\,\,\forall t_1,t_2\in \R_+,$ $\forall y_0\in \X,$
\end{itemize}
 where $G\left( t,D\right) =\underset{y\in D}{\cup }G\left(
t,y\right) ,\,\,D\subseteq \X$.

 {We denote by} $\dist_\X(C,D)=\sup_{c\in C}\inf_{d\in
D}\rho(c,d)$ the \emph{Hausdorff semidistance} between nonempty subsets $C$ and $D$ of the Polish space $\X$. Recall that
the set $\Re \subset \X$ is a \emph{global attractor} of the m-semiflow $G$ if it satisfies the following conditions: \begin{itemize}
\item[(i)] $\Re $ attracts each bounded subset $B\subset \X$, i.e.
\begin{equation}\label{eq:22}
\dist_\X(G(t,B),\Re)\to 0,\quad t\to+\infty;
\end{equation}
\item[(ii)] $\Re $ is negatively semi-invariant set, i.e. $\Re \subseteq G\left( t,\Re
\right)$ for each $t\ge 0$;
\item[(iii)] $\Re $ is the minimal set among all nonempty closed subsets $C\subseteq \X$ that satisfy (\ref{eq:22}). \end{itemize}

In this paper we examine the uniform long-time behavior of solution sets $\{\K_\tau^+\}_{\tau\ge0}$ in the strong topology of the natural phase space $H$ (as time $t\to+\infty$)  in the sense of the existence
of a compact global attractor for m-semiflow $G$ generated by the family of solution sets $\{\K_\tau^+\}_{\tau\ge0}$ and their shifts.
The following theorem is the main result of the paper.
\begin{theorem}\label{t:main}
Let assumptions (A1)--(A2) hold. Then the m-semiflow $G$, defined in (\ref{eq:5}), has a compact global attractor $\Re$ in the phase space $\X$.
\end{theorem}

\section{Proof of Theorem~\ref{t:main}}

Before the proof of Theorem~\ref{t:main} we provide the following  statement characterizing the compactness properties of the family $\K^+$ in the topology induced from $C^{\rm \rm
loc}(\R_+;H)$.

\begin{theorem}\label{teor:2}
Let assumptions (A1)--(A2) hold. Then the following two statements hold:
\begin{itemize}
\item[(a)]  for each $y\in \mathcal{K}^+$, the following estimate holds
\begin{equation}\label{eq:77}
\|y(t)\|_H^2\le\|y(0)\|_H^2 e^{-c_3t}+c_4,\quad t\ge 0,
\end{equation}
where the positive constants $c_3$ and $c_4$ do not depend on $y\in \mathcal{K}^+$ and $t\ge0$;
\item[(b)] for any bounded in $L_\infty(\mathbb{R}_+;H)$ sequence $\{y_n\}_{n\ge 1}\subset \mathcal{K}^+$  ,  there
exist an increasing sequence $\{{n_k}\}_{k\ge 1}\subseteq\mathbb{N}$ and an element $y\in \mathcal{K}^+$ such that
\begin{equation}\label{eq:*}
\|\Pi_{\tau,T}y_{n_k}-\Pi_{\tau,T}y\|_{C([\tau,T];H)}\to0,\quad k\to+\infty,
\end{equation}
for each finite time interval $[\tau,T]\subset(0,+\infty)$. If, additionally, there exists $y_0\in H$ such that $y_{n_k}(0)\to y_0$ in $H$, then $y(0)=y_0$.
\end{itemize}
\end{theorem}
\begin{proof}[Proof of Theorem~\ref{teor:2}]
Let us prove statement (a). If statement (a) holds for each $y\in \mathcal{K}_\cup^+$, then
  inequality (\ref{eq:77}) holds for each $y\in \mathcal{K}^+$, due to (\ref{eq:3}). The rest of the proof of statement (a) establishes inequality (\ref{eq:77}) for each $y\in \mathcal{K}_\cup^+$.

For an arbitrary $y\in \mathcal{K}_\cup^+$, there exist $\tau,h\ge 0$ and $z(\,\cdot\,)\in \mathcal{K}_{\tau}^{+}$ such that  $y(\,\cdot\,)=T(\tau+h)z(\,\cdot\,)$. Assumption (A1) implies the following inequality:
  \begin{equation}\label{eq:t2:1}
  \|y(t_2)\|_H^2-\|y(t_1)\|_H^2+\alpha_1\int_{t_1}^{t_2}\|y(t)\|_{V}^{p}dt\le \int_{t_1}^{t_2}c_1(t+\tau+h)dt,
  \end{equation}
  for each $t_2\ge t_1\ge 0,$ where $c_1(\cdot)$ is t.u.i. in $L_1^{\rm \rm loc}(\R_+)$. Since the embedding $V\subset H$ is compact, then this embedding is continuous. So, there exists a constant $\beta>0$ such that
   $\|b\|_{H}\le \beta \|b\|_{V}$ for each $b\in V.$
   According to (\ref{eq:t2:1}), since the inequality $a^2\le 1+a^{p}$ holds for each $a\ge 0$, then the following inequality holds:
  \begin{equation}\label{eq:t2:2}
  \|y(t_2)\|_H^2-\|y(t_1)\|_H^2+\alpha_3\int_{t_1}^{t_2}\|y(t)\|_{H}^{2}dt\le \int_{t_1}^{t_2}\left[c_1(t+\tau+h)+\alpha_3\right]dt,
  \end{equation}
for each $t_2\ge t_1\ge 0,$ where $\alpha_3=\frac{\alpha_1}{\beta^{p}} $.  Let us set
\[
\rho(t):= \|y(t)\|_H^2+\alpha_3\int_{0}^{t}\|y(s)\|_{H}^{2}ds-\int_{0}^{t}\left[c_1(s+\tau+h)+\alpha_3\right]ds,\quad t\ge 0.
\]
Inequality (\ref{eq:t2:2}) and Ball \cite[Lemma~7.1]{Ball97} yield that $\frac{d}{dt}\rho\le 0$ in $D^*((0,+\infty)),$ where $\frac{d}{dt}$ is the derivative operation in the sense of $D^*((0,+\infty))$. Thus,
\[
\frac{d}{dt}\|y(t)\|_H^2 +\alpha_3\|y(t)\|_{H}^{2}-\left[c_1(t+\tau+h)+\alpha_3\right]\le 0\mbox{ in }D^*((0,+\infty)).
\]
Therefore,
\begin{equation}\label{eq:t2:3}
\frac{d}{dt}\left[\|y(t)\|_H^2e^{\alpha_3t}\right] -e^{\alpha_3t}\left[c_1(t+\tau+h)+\alpha_3\right]\le 0\mbox{ in }D^*((0,+\infty)).
\end{equation}
Ball \cite[Lemma~7.1]{Ball97} and inequality (\ref{eq:t2:3}) imply
\begin{equation}\label{eq:t2:4}
\|y(t_2)\|_H^2\le \|y(t_1)\|_H^2e^{-\alpha_3(t_2-t_1)} +\int_{t_1}^{t_2}e^{-\alpha_3(t_2-t)}\left[c_1(t+\tau+h)+\alpha_3\right]dt,
\end{equation}
for each $t_2\ge t_1\ge 0.$ Therefore,
\[
\begin{aligned}
\|y(t_2)\|_H^2\le &\|y(t_1)\|_H^2e^{-\alpha_3(t_2-t_1)} +\int_{t_1}^{t_2}e^{-\alpha_3(t_2-t)}\left[c_1(t+\tau+h)+\alpha_3\right]dt\le \\
&\|y(t_1)\|_H^2e^{-\alpha_3(t_2-t_1)} +1+\int_{t_1+\tau+h}^{t_2+\tau+h}e^{-\alpha_3(t_2-t+\tau+h)}c_1(t)dt\le \\
& \|y(t_1)\|_H^2e^{-\alpha_3(t_2-t_1)} +1+\frac{K}{\alpha_3} {+}\\
& \int_{t_1+\tau+h}^{t_2+\tau+h}e^{-\alpha_3(t_2-t+\tau+h)}|c_1(t)|\I{\{|c_1(t)|\ge K\}}dt,
\end{aligned}
\]
for each $K>0$, $t_2\ge t_1\ge 0.$
Since the function $c_1:\R_+\to\R_+$  is t.u.i. in $L_1^{\rm \rm loc}(\R_+)$ (see assumption (A1)), then there exists $K_0>0$ such that
\[
\sup_{t\ge0}\int_{t}^{t+1}|c_1(s)|\I{\{|c_1(s)|\ge K_0\}}ds\le 1.
\]
Thus,
\[\begin{aligned}
\|y(t_2)\|_H^2\le &\|y(t_1)\|_H^2e^{-\alpha_3(t_2-t_1)} +1+\frac{K_0}{\alpha_3}+{e^{\alpha_3}+1},
\end{aligned}\]
that yields estimate (\ref{eq:77}) with $c_3:=\alpha_3$ and $c_4:=1+\frac{K_0}{\alpha_3}+e^{\alpha_3}+1$, where the positive constants $c_3$ and $c_4$ do not depend on $y\in \mathcal{K}^+$ and $t\ge0$.

Let us prove statement (b). Let $\{y_n\}_{n\ge 1}\subset \mathcal{K}^+$ be an arbitrary sequence that is bounded in
$L_\infty(\mathbb{R}_+;H)$. Since $\mathcal{K}_\cup^+$ is the dense set in a Polish space
$\mathcal{K}^+$ endowed with the topology induced from $C^{\rm \rm
loc}(\R_+;H)$, then for each $n\ge 1$ there exists $u_n\in \mathcal{K}_\cup^+$ such that
\begin{equation}\label{eq:un}
\rho_{C^{\rm \rm
loc}(\R_+;H)}(y_n,u_n)\le \frac1n, \mbox{ for each }n\ge 1.
\end{equation}
Note that a priori estimate (\ref{eq:77}) provides that the sequence $\{u_n\}_{n\ge 1}$ is bounded in
$L_\infty(\mathbb{R}_+;H)$. Therefore, the rest of the proof establishes statement (b) for the sequence $\{u_n\}_{n\ge 1}$.

Let us fix $n\ge 1$. Formula (\ref{eq:2}) provides the existence of
$\tau_n,h_n\ge0$ and  $z_n(\,\cdot\,)\in \mathcal{K}_{\tau_n}^+$ such that $u_n(\,\cdot\,)=z_n(\,\cdot\,+\tau_n+h_n)$. Then,
assumptions (A1) and (A2) yield
\begin{equation}\label{eq:stb1}
\begin{aligned}
\|{u_n}(t_2)\|_H^2-\|{u_n}(t_1)\|_H^2+\alpha_1\int_{t_1}^{t_2}\|{u_n}(t)\|_{V}^{p}dt\le \int_{t_1}^{t_2}c_1(t+\tau_n+h_n)dt,
\\
\int_{t_1}^{t_2}\|{u_n}'(t)\|_{{V^*}}^{q}dt\le \alpha_2 \int_{t_1}^{t_2}\|{u_n}(t)\|_{V}^{p}dt+ \int_{t_1}^{t_2}c_2(t+\tau_n+h_n)dt,
\end{aligned}
\end{equation}
for each $t_2\ge t_1\ge 0$ and $n\ge 1$.

We remark that
\begin{equation}\label{eq:stb2}
\sup_{n\ge 1}\int_{t_1}^{t_2}|c_1(t+\tau_n+h_n)|dt<\infty\mbox{ and }\sup_{n\ge 1}\int_{t_1}^{t_2}|c_2(t+\tau_n+h_n)|dt<\infty,
\end{equation}
for each $t_2\ge t_1\ge 0$, since the functions $c_1,c_2:\R_+\to\R_+$ are t.u.i. in $L_1^{\rm \rm loc}(\R_+)$.

Formulae (\ref{eq:stb1}) and (\ref{eq:stb2}) imply that the sequence $\{u_n\}_{n\ge 1}$ is bounded in $W^{\loc}(\R_+)$. Thus, Banach--Alaoglu theorem and Zgurovsky et al. \cite[Theorems~1.16 and 1.21]{ZMK2}
yield that there exist an increasing sequence $\{{n_k}\}_{k\ge 1}\subseteq\mathbb{N}$ and elements $y\in W^{loc}(\R_+)\subset C^{\rm \rm loc}(\R_+;H)$ and $\bar{c}_1
\in L_1^{\rm \rm loc}(\R_+)$ such that
\begin{equation}\label{eq:88}
\begin{array}{ll}
u_{n_k}\to y &\mbox{weakly in  }L_{p}^{\rm \rm loc}(\R_+;V),\\
u_{n_k}'\to y' &\mbox{weakly in  }L_{q}^{\rm \rm loc}(\R_+;{V^*}),\\
u_{n_k}\to y &\mbox{weakly  in  }C^{\rm \rm loc}(\R_+;H),\\
u_{n_k}(t)\to y(t) &\mbox{in  }H\mbox{ for a.e. }t>0,\\
c_{1}(\,\cdot\,+\tau_{n_k}+h_{n_k})\to\bar{c}_1&\mbox{weakly in  }L_1^{\rm \rm loc}(\R_+),\quad k\to\infty,
\end{array}
\end{equation}
where the last convergence holds due to the fact that $c_1
\in L_1^{\rm \rm loc}(\R_+)$  is t.u.i. in $L_1^{\rm \rm loc}(\R_+)$. According to (\ref{eq:88}), we can pass to the limit in (\ref{eq:1a}). So, we obtain that $y$ satisfies (\ref{eq:1a}).

We consider the continuous and nonincreasing (by assumption (A1))
functions on $\R_+$:
\begin{equation}\label{eq:8.31}
\begin{aligned}
&J_k(t)=\|u_{n_k}(t)\|_H^2 - \int_0^t c_{1}(s+\tau_{n_k}+h_{n_k}) ds,\\
&J(t)=\|y(t)\|_H^2 -\int_0^t \bar{c}_1(s) ds,\quad k\ge 1;
\end{aligned}
\end{equation}
cf. Kapustyan and Valero et al. \cite{KapVal09}. The  last two statements in (\ref{eq:88}) imply
\begin{equation}\label{chapter2eq:8.4}
J_k(t)\to J(t),\mbox{ as }k\to +\infty,\mbox{ for a.e. }t>0.
\end{equation}

Similarly to Zgurovsky et al. \cite[p.~57]{ZMK3} (see the book and references therein) we show that
(\ref{eq:*}) holds. By contradiction suppose the existence of
a positive constant $L>0$, a finite interval $[\tau,T]\subset
(0,+\infty)$, and a subsequence $\{u_{k_j}\}_{j\ge
1}\subseteq\{u_{n_k}\}_{k\ge1}$ such that
\[
 \max_{t\in[\tau,T]}\|u_{k_j}(t)-y(t)\|_H= \|u_{k_j}(t_j)-y(t_j)\|_H\ge L,
\]
for each $j\ge 1.$
Suppose also that $t_j\to t_0\in [\tau,T]$, as $j\to+\infty$. Continuity of $\Pi_{\tau,T}y:[\tau,T]\to H$ implies
\begin{equation}\label{eq:**}
\isup_{j\to+\infty}\|u_{k_j}(t_j)-y(t_0)\|_H\ge L.
\end{equation}

On the other hand, we prove that
\begin{equation}\label{eq:19}
u_{k_j}(t_j)\to y(t_0) \mbox{ in } H,\,\, j\to +\infty.
\end{equation}

For this purpose we firstly note that from (\ref{eq:88}) we have
\begin{equation}\label{eq:20}
u_{k_j}(t_j)\to y(t_0)\mbox{ weakly in } H,\quad j\to+\infty.
\end{equation}

Secondly we prove that
\begin{equation}\label{eq:21}
\vsup_{j\to+\infty}\|u_{k_j}(t_j)\|_H\le\|y(t_0)\|_H.
\end{equation}
We consider the continuous nonincreasing functions $J$ and $J_{k_j}$, $j\ge1$, defined in (\ref{eq:8.31}). Let us fix an arbitrary $\varepsilon
 >0$. The continuity of $J$ and (\ref{chapter2eq:8.4}) provide the existence of $\bar{t}\in(\tau,t_0)$ such that
$\lim_{j\to\infty}J_{k_j}(\bar{t})=J(\bar{t})$  and  $|J(\bar{t})-J(t_0)|<\varepsilon$. Then,
\[
J_{k_j}(t_j)-J(t_0)\le |J_{k_j}(\bar{t})-J(\bar{t})|+|J(\bar{t})-J(t_0)|\le |J_{k_j}(\bar{t})-J(\bar{t})|+\varepsilon,
\]
for rather large $j\ge 1$. Thus, $\vsup_{j\to+\infty}J_{k_j}(t_j)\le
 J(t_0)$ and inequality (\ref{eq:21}) holds.

Thirdly note that the convergence (\ref{eq:19}) holds due to
(\ref{eq:20}), (\ref{eq:21}); cf. Gajewski et al.~\cite[Chapter~I]{GGZ}.  Finally, we remark that statement
(\ref{eq:19}) contradicts assumption (\ref{eq:**}). Therefore, according to (\ref{eq:un}),
the first statement of the theorem holds for each sequence
$\{y_n\}_{n\ge 1}\subset \mathcal{K}^+$.

To finish the proof of  statement (b) we note that if, additionally, there exists $y_0\in H$ such that $y_{n_k}(0)\to y_0$ in $H$, then, according to the third convergence in (\ref{eq:88}), $y(0)=y_0$.
\end{proof}

\medskip

Let us provide the proof of the main result.

\medskip

\begin{proof}[Proof of Theorem~\ref{t:main}]
Theorem ~\ref{teor:2} implies the following properties for the  {m-semi\-flow~$G$}, defined in (\ref{eq:5}):
\begin{itemize}
\item[(a)] for each $t\ge 0$ the mapping $G(t,\,\cdot\,):\X\to 2^{\X}\setminus\{\emptyset\}$ has a closed graph;
\item[(b)] for each $t\ge0$ and $y_0\in \X$ the set $G(t,y_0)$ is compact in $\X$;
\item[(c)] the set $G(1,\tilde{C})$, where $\tilde{C}:=\{z\in \X\,:\, \|z\|_H^2< c_4+1 \}$, is precompact and attracts each bounded subset $C\subset \X$.
\end{itemize}
Indeed, property (a) follows from Theorem~\ref{teor:2} (see formulae (\ref{eq:3}) and (\ref{eq:5}));
 property (b) directly follows from (a) and Theorem~\ref{teor:2}(b); property (c) holds, since $G(1,\tilde{C})$ is precompact in $\X$ (Theorem~\ref{teor:2}(b) and formula~(\ref{eq:5})) and the following inequalities and  {equality hold}:
\begin{equation*}
\dist_\X(G(t,C),G(1,\tilde{C}))\le \dist_\X(G(1,G(t-1,C)),G(1,\tilde{C}))\le
\end{equation*}
\begin{equation*}
\dist_\X(G(1,\tilde{C}),G(1,\tilde{C}))=0,
\end{equation*}
for sufficiently large $t$.

According to properties (a)--(c), Mel'nik and Valero \cite[Theorems~1, 2, Remark~2, Proposition~1]{MeVa1997} yields that
 the m-semiflow $G$ has a compact global attractor $\Re$ in the phase space $\X$.
\end{proof}

\section{Applications}\label{sec:app}
In the following three examples we examine the uniform global attractor for the family of solution sets $\{\mathcal{K}_{\tau}^{+}\}$ generated by particular evolution problems. In all the cases we assume that
\[
\forall z\in H \,\ \forall \tau\ge 0 \,\ \exists y\in \mathcal{K}_{\tau}^{+} \mbox{ such that } y(\tau)=z.
\]
This assumption guarantees the equality $\X=H$.

\textbf{Example 4.1}
{\rm (Autonomous evolution problem) Let $\{\mathcal{K}_{\tau}^{+}\}$ be a family of solutions for an autonomous problem on $[\tau, +\infty)$, $\tau\ge 0$. Then we have:
\begin{equation}\label{eq49}
\forall h\ge 0 \,\,\ T(h)\mathcal{K}_{0}^{+} \subset \mathcal{K}_{0}^{+};
\end{equation}
\begin{equation}\label{eq50}
\forall \tau\ge 0 \,\ \forall y\in \mathcal{K}_{\tau}^{+} \,\ y(\cdot +\tau) \in \mathcal{K}_{0}^{+}.
\end{equation}
So, $\mathcal{K}_\cup^+=\mathcal{K}_{0}^{+}$. If additionally we have that
\begin{equation}\label{eq51}
\mathcal{K}_{0}^{+} \mbox{ is closed in } C^{\rm \rm loc}(\mathbb{R}_+;H),
\end{equation}
then
$$
\mathcal{K}^{+}=\mathcal{K}_{0}^{+}.
$$
It implies that the m-semiflow $G$ (defined by (\ref{eq:5})) is a classical multivalued semigroup generated by an autonomous evolution problem.
}

\textbf{Example 4.2}
{\rm  {(}Non-autonomous evolution problem) Let $\{\mathcal{K}_{\tau}^{+}\}$ be a family of solutions for non-autonomous problem on $[\tau, +\infty)$, $\tau\ge 0$, and the following condition holds:
\begin{equation}\label{eq52}
\forall s\ge\tau\ge 0 \,\ \forall y\in \mathcal{K}_{\tau}^{+} \,\,\,\ \Pi_{s,+\infty} y(\cdot)\in \mathcal{K}_{s}^{+}.
\end{equation}
Then, according to Kapustyan et al. \cite{Kaskapvalz}, formula
\begin{equation}\label{eq53}
U(t,\tau,z)=\{y(t) \,: \, y(\cdot)\in \mathcal{K}_{\tau}^{+}, \,\ y(\tau)=z\}
\end{equation}
defines a m-semiprocess, that is
\[
\forall t\ge s \ge \tau \,\,\,\  U(t,\tau,z) \subset U(t,s,U(s,\tau,z)).
\]
One of the most important objects for m-semiprocess (\ref{eq53}) is uniform global attractor;  Chepyzhov and Vishik\cite{VChRI}, Kapustyan et al. \cite{KKV}, Zgurovsky et al. \cite{ZMK3}. It is a set $\Theta$ such that for every bounded subset $C\subset H$
\begin{equation}\label{eq54}
\sup_{\tau\ge0}\dist_H (U(t+\tau,\tau,C),\Theta) \to 0, \,\,\ t\to\infty,
\end{equation}
and $\Theta$ is minimal among all closed sets satisfying this property. Then under  {assumptions} (A1), (A2) and  from (\ref{eq52}) it follows that the m-semiprocess (\ref{eq53}) has the compact uniform global attractor $\Theta \subseteq \Re$, where
$\Re$ is the global attractor for the m-semiflow (\ref{eq:5}).

Indeed,
\begin{equation}\label{eq55}
\forall t\ge\tau\ge 0 \,\ \forall z\in H \,\,\,\,\ U(t+\tau,\tau,z) \subset G(t,z).
\end{equation}
So, if $\Re$ is a compact global attractor for the m-semiflow $G$ then, according to Kapustyan et al. \cite{KKV},  there exists a compact uniform  global attractor $\Theta$
for m-semiprocess $U$ and, moreover, $\Theta\subset \Re$.}

In the following example we examine the existence of uniform global attractor for non-autonomous differential-operator inclusion. The uniform trajectory attractors for classes of non-autonomous inclusions and equations were proved to exist in Zgurovsky and Kasyanov \cite{dcds9} (see also Gorban et al. \cite{GKK}).

\textbf{Example 4.3}
{\rm (Non-autonomous differential-operator inclusion) For the multivalued map
$A:\R_+\times V\rightarrow 2^{V^*\setminus\{\emptyset\}}$ we consider the problem of long-time behavior of all globally defined weak solutions for non-autonomous evolution
inclusion
\begin{equation}\label{eq:1}
y'(t)+A(t,y(t))\ni \bar{0},
\end{equation}
as $t\to+\infty$. Let $\langle\cdot,\cdot\rangle_V:V^*\times V\to\mathbb{R}$ be the pairing in $V^*\times V$, that coincides on $H\times V$ with the inner
product $(\cdot,\cdot)$ in the Hilbert space $H$.

We note that Problem (\ref{eq:1}) arises in many important models for distributed parameter control problems and that large class of identification
problems enter this formulation. Let us indicate a problem which is one of the motivations for the study of the non-autonomous evolution inclusion (\ref{eq:1})
(see, for example, Mig${\rm\acute{o}}$rski and Ochal \cite{JGO2000}; Zgurovsky et al. \cite{ZMK3} and references therein). In a subset $\Omega$ of
$\mathbb{R}^3,$ we consider the nonstationary heat conduction equation
$$
\frac{\partial y}{\partial t}-\triangle y=f\,\,\mbox{in}\,\,\Omega\times (0,+\infty)
$$
with initial conditions and suitable boundary ones. Here $y=y(x,t)$ represents the temperature at the point $x\in \Omega$ and time $t>0.$ It is supposed
that $f={f_1}+f_2,$ where ${f_2}$ is given and ${f_1}$ is a known function of the temperature of the form
$$
-f_1(x,t)\in \partial j(x,t,y(x,t))\,\,\mbox{a.e.}\,\,(x,t)\in \Omega\times (0,+\infty).
$$
Here $\partial j(x,t,\xi)$ denotes generalized gradient of Clarke (see Clarke \cite{chapter2cl}) with respect to the last variable of a function
$j:\Omega\times \mathbb{R_+}\times \mathbb{R}\to \mathbb{R}$ which is assumed to be locally Lipschitz in $\xi$ (cf. Mig${\rm\acute{o}}$rski and Ochal \cite{JGO2000} and
references therein). The multivalued function $\partial j(x,t,\cdot):\mathbb{R}\to 2^{\mathbb{R}}$ is generally nonmonotone and it includes the vertical
jumps. In a physicist's language it means that the law is characterized by the generalized gradient of a nonsmooth potential $j$ (cf. Panagiotopoulos
\cite{Pana}). Models of physical interest includes also the next (see, for  {example}, Balibrea et al. \cite{introPersp} and references therein): a model of
combustion in porous media; a model of conduction of electrical impulses in nerve axons; a climate energy balance model; etc.

Let the following  {assumptions} hold:
\begin{itemize}
\item[(H1)](\emph{Growth condition}) There exist a  t.u.i. in $L_1^{\rm loc}(\R_+)$ function $c_1:\R_+\to\R_+$ and a constant $c_2>0$  such that
$\|d\|_{V^*}^{q}\le c_1(t)+c_2\|u\|_V^{p}$ for any $u\in V$, $d\in A(t,u)$, and a.e. $t> 0$;
\item[(H2)](\emph{Sign condition}) There exist a constant $\alpha>0$ and a  t.u.i. in $L_1^{\rm loc}(\R_+)$ function $\beta:\R_+\to\R_+$ such that
$\langle d,u\rangle_V\ge \alpha\|u\|_V^p-\beta(t)$ for any $u\in V$, $d\in A(t,u)$, and a.e. $t> 0$;
\item[(H3)](\emph{Strong measurability}) If $C\subseteq V^*$ is a closed set, then the set $\{(t,u)\in(0,+\infty)\times V\,:\, A(t,u)\cap C\ne\emptyset\}$
is a Borel subset in $(0,+\infty)\times V$;
\item[(H4)](\emph{Pointwise pseudomonotonicity}) Let for a.e. $t> 0$ the following two assumptions hold: \begin{itemize} \item[]a) for every $u\in V$ the set $A(t,u)$
is nonempty, convex, and weakly compact one in $V^*$; \item[]b) if a sequence $\{u_n\}_{n\ge 1}$ converges weakly in $V$ towards $u\in V$ as $n\to+\infty$,
$d_n\in A(t,u_n)$ for any $n\ge 1$, and $\slim_{n\to+\infty}\langle d_n,u_n-u\rangle_V\le 0$, then for any $\omega \in V$ there exists $d(\omega)\in
A(t,u)$ such that \[\ilim_{n\to+\infty} \langle d_n, u_n-\omega\rangle_V\ge \langle d(\omega), u-\omega\rangle_V.\]
\end{itemize}
\end{itemize}

Let $0\le \tau<T<+\infty$. As a \textit{weak solution} of evolution inclusion (\ref{eq:1}) on the interval $[\tau,T]$ we consider an element $u(\cdot)$ of
the space $L_p(\tau,T;V)$ such that for some $d(\cdot)\in L_q(\tau,T;V^*)$ it is fulfilled:
\begin{equation}\label{Warningeq:3}
-\int_\tau^T(\xi'(t),y(t))dt+ \int_\tau^T\langle d(t),\xi(t)\rangle_V dt= 0\quad \forall \xi\in C_0^\infty([\tau,T];V),
\end{equation}
and $d(t)\in A(t,y(t))$ for a.e. $t\in (\tau,T)$.
For fixed nonnegative $\tau$ and $T$, $\tau<T$, let us consider
\begin{equation*}
X_{\tau,T}=L_p(\tau,T;V),\quad X_{\tau,T}^*=L_q(\tau,T;V^*),
\end{equation*}
\begin{equation*}
W_{\tau,T}=\{y\in X_{\tau,T}\ |\ y'\in X_{\tau,T}^*\},\quad
\mathcal{A}_{\tau,T}:X_{\tau,T}\rightarrow 2^{X_{\tau,T}^*}\setminus\{\emptyset\},
\end{equation*}
\begin{equation*}
\mathcal{A}_{\tau,T}(y)=\{d\in X_{\tau,T}^*\,|\, d(t)\in A(t,y(t))\mbox{ for a.e. }t\in(\tau,T)\},
\end{equation*}
where $y'$ is a derivative of an element $u\in X_{\tau,T}$ in the sense of $\mathcal{D}([\tau,T];V^*)$ (see, for example, Gajewski, Gr\"{o}ger, and
Zacharias \cite[Definition~IV.1.10]{GGZ}). Gajewski, Gr\"{o}ger, and Zacharias \cite[Theorem~IV.1.17]{GGZ} provide that the embedding $W_{\tau,T}\subset
C([\tau,T];H)$ is continuous and dense. Moreover,
\begin{equation}\label{Warningeq:4}
(u(T),v(T))-(u(\tau),v(\tau))=\int_\tau^T\Bigl[ \langle u'(t),v(t)\rangle_V+\langle v'(t),u(t)\rangle_V\Bigr]dt,
\end{equation}
for any $u,v\in W_{\tau,T}$.

Mig${\rm\acute{o}}$rski \cite[Lemma~7, p.~516]{Mig} (see the paper and references therein) and the assumptions above provide that the multivalued mapping
$\mathcal{A}_{\tau,T}:X_{\tau,T}\rightarrow 2^{X_{\tau,T}^*}\setminus\{\emptyset\}$ satisfies the listed below  {properties}:
\begin{itemize}
\item[(P1)]There exists a positive constant $C_1=C_1(\tau,T)$ such that $\|d\|_{X_{\tau,T}^*}\le C_1(1+\|y\|_{X_{\tau,T}}^{p-1})$ for any $y\in X_{\tau,T}$ and $d\in \mathcal{A}_{\tau,T}(y)$;
\item[(P2)]There exist positive constants $C_2=C_2(\tau,T)$ and $C_3=C_3(\tau,T)$ such that $\langle
d,y\rangle_{X_{\tau,T}}\ge C_2\|y\|_{X_{\tau,T}}^p-C_3$ for any $y\in X_{\tau,T}$ and  $d\in \mathcal{A}_{\tau,T}(y)$;
\item[(P3)]The
multivalued mapping $\mathcal{A}_{\tau,T}:X_{\tau,T}\rightarrow 2^{X_{\tau,T}^*}\setminus\{\emptyset\}$ is (generalized) pseudomonotone on $W_{\tau,T}$, i.e.
\begin{itemize} \item[]a) for every
$y\in X_{\tau,T}$ the set $\mathcal{A}_{\tau,T}(y)$ is a nonempty, convex and weakly compact one in $X_{\tau,T}^*$;
\item[]b) $\mathcal{A}_{\tau,T}$ is upper
semi-continuous from every finite dimensional subspace $X_{\tau,T}$ into $X_{\tau,T}^*$ endowed with the weak topology;
\item[]c) if a sequence $\{y_n,d_n\}_{n\ge
1}\subset W_{\tau,T} \times X_{\tau,T}^*$ converges weakly in $W_{\tau,T}\times X_{\tau,T}^*$ towards $(y,d)\in W_{\tau,T}\times X_{\tau,T}^*$, $d_n\in
\mathcal{A}_{\tau,T}(y_n)$ for any $n\ge 1$, and $\slim_{n\to+\infty}\langle d_n,y_n-y\rangle_{X_{\tau,T}}\le 0$, then $d\in
\mathcal{A}_{\tau,T}(y)$ and \\$\lim_{n\to+\infty} \langle d_n, y_n\rangle_{X_{\tau,T}}=\langle d, y\rangle_{X_{\tau,T}}$.
\end{itemize}
\end{itemize}

Formula (\ref{Warningeq:3}) and the definition of the derivative for an element from $\mathcal{D}([\tau,T];V^*)$ yield that each weak solution $y\in X_{\tau,T}$ of
Problem (\ref{eq:1}) on $[\tau,T]$ belongs to the space $W_{\tau,T}$ and $y'+\mathcal{A}_{\tau,T}(y)\ni \bar{0}$.   {On the contrary, suppose that} $y\in W_{\tau,T}$
satisfies the last inclusion, then $y$ is a weak solution of Problem (\ref{eq:1}) on $[\tau,T]$.

Assumption (H1),  {properties} (P1)--(P3), and Denkowski, Mig${\rm\acute{o}}$rski, and Papageorgiou \cite[Theorem~1.3.73]{DMP} (see also Zgurovsky, Mel'nik, and
Kasyanov \cite[Chapter~2]{ZMK2} and references therein) provide the existence of a weak solution of Cauchy problem (\ref{eq:1}) with initial data
$y(\tau)=y^{(\tau)}$ on the interval $[\tau,T]$, for any $y^{(\tau)}\in H$.

For fixed $\tau$ and $T$, such that $0\le\tau<T<+\infty$, we denote
\[
\mathcal{D}_{\tau,T}(y^{(\tau)})=\{y(\cdot)\ | \ y\mbox{ is a weak solution of (\ref{eq:1}) on } [\tau,T], \ y(\tau)=y^{(\tau)}\},\quad y^{(\tau)}\in H.
\]
We remark that $\mathcal{D}_{\tau,T}(y^{(\tau)})\ne \emptyset$ and $\mathcal{D}_{\tau,T}(y^{(\tau)})\subset W_{\tau,T}$, if $0\le\tau<T<+\infty$ and
$y^{(\tau)}\in H$. Moreover, the concatenation of weak solutions of Problem (\ref{eq:1}) is a weak solutions too, i.e. if $0\le\tau<t<T$, $y^{(\tau)}\in H$,
$y(\cdot)\in \mathcal{D}_{\tau,t}(y^{(\tau)})$, and $v(\cdot)\in \mathcal{D}_{t,T}(y(t))$, then
\[z(s)=\left\{
\begin{array}{ll}
y(s),&s\in[\tau,t],\\
v(s),&s\in[t,T],
\end{array}\right.\] belongs to $\mathcal{D}_{\tau,T}(y^{(\tau)})$; cf. Zgurovsky et al. \cite[pp.~55--56]{ZMK3}.

Gronwall's lemma provides that for any finite time interval $[\tau,T]\subset \R_+$ each weak solution $y$ of Problem (\ref{eq:1}) on $[\tau,T]$ satisfies the
estimates
\begin{equation}\label{Warningeq:6}
\|y(t)\|_H^2 -2 \int_0^t \beta(\xi)d\xi+ 2\alpha \int_s^t\|y(\xi)\|_V^p d\xi \le \|y(s)\|_H^2- 2 \int_0^s \beta(\xi)d\xi,
\end{equation}
\begin{equation}\label{eq:7}
\|y(t)\|_H^2\le\|y(s)\|_H^2 e^{-2\alpha\gamma (t-s)}+2 \int_s^t (\beta(\xi)+\alpha\gamma)e^{-2\alpha\gamma (t-\xi)}d\xi,
\end{equation}
where $t,s\in[\tau,T]$, $t\ge s$; $\gamma>0$ is a constant such that $\gamma\|u\|_{H}^p\le \|u\|_V^p$ for any $u\in V$; cf. Zgurovsky et al.
\cite[p.~56]{ZMK3}. In the proof of (\ref{eq:7}) we used the inequality $\|u\|_{H}^2-1\le \|u\|_{H}^p$ for any $u\in H$.

Therefore, any weak solution $y$ of Problem (\ref{eq:1}) on a finite time interval $[\tau,T]\subset \R_+$ can be extended to a global one, defined on
$[\tau,+\infty)$. For arbitrary $\tau \ge 0$ and $y^{(\tau)}\in H$ let $\mathcal{D}_\tau(y^{(\tau)})$ be the set of all weak solutions (defined on
$[\tau,+\infty)$) of Problem (\ref{eq:1}) with initial data $y(\tau)=y^{(\tau)}$. Let us consider the family $\mathcal{K}_\tau^+=\cup_{y^{(\tau)}\in
H}\mathcal{D}_\tau(y^{(\tau)})$ of all weak solutions of Problem (\ref{eq:1}) defined on the semi-infinite time interval $[\tau,+\infty)$.

Properties (P1)--(P2) imply assumptions (A1) and (A2). Therefore, Theorem~\ref{t:main} yields that the m-semiflow $G$, defined in (\ref{eq:5}), has a compact global attractor $\Re$ in the phase space $H.$
}
\section*{Acknowledgments} Authors thank Prof. Oleksiy V. Kapustyan for useful discussions during the preparation of this paper. Authors thank the anonymous {referees for their} constructive suggestions and remarks.
The research was partially supported
by the National Academy of Sciences of Ukraine
under grant 2284/15 and by Grant of the President of Ukraine GP/F61/017.



\end{document}